\documentclass[12pt]{amsart}

\usepackage{amsfonts, amsmath, amscd}
\usepackage[psamsfonts]{amssymb}
\usepackage{amssymb}
\usepackage{enumerate}
\usepackage{pb-diagram}
\usepackage[usenames]{color}
\usepackage{amssymb}
\usepackage{amsfonts}
\usepackage{graphicx}
\usepackage{amscd}
\newtheorem{theorem}{Theorem}[section]
\newtheorem{corollary}[theorem]{Corollary}
\newtheorem{lemma}[theorem]{Lemma}
\newtheorem{proposition}[theorem]{Proposition}
\newtheorem{problem}[theorem]{Problem}

\theoremstyle{definition}

\newtheorem{remark}[theorem]{Remark}

\def\supp{{\mathrm {supp}}\,}
\def\span{{\mathrm {span}}\,}

\title[A note on Banach spaces]
{A note on Banach spaces $E$ admitting a continuous map from $C_p(X)$ onto $E_{w}$}
\author{Jerzy K\c{a}kol, Arkady Leiderman, Artur Michalak}

\address{Faculty of Mathematics and Informatics, A. Mickiewicz University,
61-614 Pozna\'{n}, Poland and Institute of Mathematics Czech Academy of Sciences, Prague, Czech Republic}
\email{kakol@amu.edu.pl}

\address{Department of Mathematics, Ben-Gurion University of the Negev, Beer Sheva, P.O.B. 653, Israel}
\email{arkady@math.bgu.ac.il}

\address{Faculty of Mathematics and Informatics, A. Mickiewicz University,
61-614 Pozna\'{n}, Poland}
\email{michalak@amu.edu.pl}

\date{\today}

\begin{document}

\begin{abstract}
$C_p(X)$ denotes the space of continuous real-valued functions on a Tychonoff space $X$ endowed with the topology of pointwise convergence.
A Banach space $E$ equipped with the weak topology is denoted by $E_{w}$.
It is unknown whether $C_p(K)$ and $C(L)_{w}$ can be homeomorphic for infinite compact spaces $K$ and $L$ \cite{Krupski-1}, \cite{Krupski-2}.
In this paper we deal with a more general question: what are the Banach spaces $E$ which admit certain continuous surjective mappings $T: C_p(X) \to E_{w}$
for an infinite Tychonoff space $X$?

First, we prove that if $T$ is linear and sequentially continuous, then the Banach space $E$ must be finite-dimensional,
 thereby resolving an open problem posed in \cite{Kakol-Leiderman}.
Second, we show that if there exists a homeomorphism $T: C_p(X) \to E_{w}$ for some infinite Tychonoff space $X$ and a Banach space $E$,
then (a) $X$ is a countable union of compact sets $X_n, n \in \omega$, where at least one component $X_n$ is non-scattered;
(b) $E$ necessarily contains an isomorphic copy of the Banach space $\ell_{1}$.
\end{abstract}

\thanks{The first named author
 is supported by the GA\v{C}R project 20-22230L and RVO: 67985840.}

\maketitle

\section{Introduction}\label{intro}

For a locally convex  space $E$ by $E_w$ we denote the space $E$ endowed with the weak topology $w=\sigma(E,E')$ of $E$, where $E'$ means the topological dual of $E$.
All topological spaces in the paper are assumed to be Tychonoff.
For a Tychonoff space $X$ let $C_{p}(X)$ be the space of continuous real-valued functions $C(X)$ endowed with the pointwise convergence topology.
If $X$ is a compact space, then $C(X)$ means the Banach space equipped with the $\sup$-norm.

In \cite{Krupski-1} M. Krupski asked the following question:
\begin{problem} \label{question1} 
 Suppose that $K$ is an infinite
(metrizable) compact space. Can $C(K)_w$ and $C_p(K)$ be homeomorphic?
\end{problem}

The main result of \cite{Krupski-1} shows that the
answer is "no", provided $K$ is an infinite metrizable compact $C$-space (in particular, if $K$ is
 any infinite metrizable finite-dimensional compact space).

A more general problem was posed in \cite{Krupski-2}:
\begin{problem} \label{question2}
Let $K$ and $L$ be infinite compact spaces. Can it happen that $C_p(L)$and $C(K)_w$ are homeomorphic?
\end{problem}

Both Problems \ref{question1}, \ref{question2} remain open, although in some cases the answer is known to be negative.
 For example,  the most easy examples of compact spaces $K$ such that $C_p(K)$ and $C(K)_w$ are not homeomorphic,
 as already observed in \cite[(D), p. 648]{Krupski-2}, are scattered compact spaces.
The reason is the following: \emph{the space $C(K)_w$  is not Fr\'echet-Urysohn for every infinite compact space $K$}. 
Recall that a topological space $X$ is called {\it scattered} if every non-empty subspace of $X$ contains an isolated point,
and $C_p(K)$  is Fr\'echet-Urysohn for any scattered compact space $K$.
 We refer the reader to articles \cite{Krupski-1}, \cite{Krupski-2} (and references  therein) discussing these problems and providing more concrete 
 examples of compact spaces $K$ and $L$ such that $C_p(K)$ and $C(L)_{w}$ are not homeomorphic. 
We recommend also the paper \cite{Marciszewski} which surveys a substantial progress in the study of various types of homeomorphisms between the function spaces $C_p(X)$.

One of the starting points of our research is the following result of M. Krupski and W. Marciszewski:  
for any infinite compact spaces $K$ and $L$ a homeomorphism $T :C(K)_{w} \to C_p(L)$ which is in addition uniformly continuous, does not exist
(see \cite[Proposition 3.1]{Krupski-2}). 
In particular, there is no a linear homeomorphism between function spaces $C(K)_{w}$ and $C_p(L)$ \cite[Corollary 3.2]{Krupski-2}.

This short summary makes it clear our motivation for a formulation of the following "linear" variant of the above Problem \ref{question1}.

\begin{problem}[\cite{Kakol-Leiderman}]\label{pro-1}
Does there exist an infinite compact space $X$ admitting a continuous linear surjection $T:C_{p}(X)\rightarrow C(X)_{w}$?
\end{problem}

It has been proved earlier that no infinite metrizable compact $C$-space $X$ admits such a mapping \cite{Kakol-Leiderman}.

In this note we suggest to consider analogous questions in a more general framework. 
It was H.H. Corson who initiated the investigation aiming to find criteria or
techniques which can be used to determine whether or not a given Banach
space $E$ under its weak topology has any of the usual topological properties (see \cite{Corson}).
The next problem combines together two lines of research: both $E_w$ and $C_p(X)$.

\begin{problem}\label{pro-2}
Let $E$ be a Banach space which admits certain continuous surjective mappings $T: C_p(X) \to E_{w}$
for an infinite Tychonoff space $X$.
1) Characterize such $E$; 2) What are the possible restrictions on $X$?
\end{problem}

Thus, this paper continues several lines of research initiated in the recent papers
\cite{Gab}, \cite{Gab2}, \cite{Gab3}, \cite{Kakol-Leiderman}, \cite{KalMol}, \cite{Krupski-1}, \cite{Krupski-2}.

Recall that a mapping between topological spaces is {\it sequentially continuous} if it sends converging sequences to converging sequences. 
The following statement proved in Section \ref{section1}, which is one of the main results of our paper, gives a negative answer to Problem \ref{pro-1} in a very strong form.

\begin{theorem}\label{Theor:general}
Let $X$ be any Tychonoff space and let $E$ be a Banach space.
Then every sequentially continuous linear operator $T:C_p(X)\to E_w$ has a finite-dimensional range.
\end{theorem}

In Section \ref{section2} we consider another particular version of Problem \ref{pro-2}, imposing on $T$ a requirement to be a (non-linear) homeomorphism. 

\begin{remark}\label{remark1}
Let $E$ be a Banach space, and $X$ be an infinite Tychonoff space.
If there is a homeomorphism $T: C_p(X) \to E_{w}$, then $E$ cannot be reflexive.
This is because $E_w$ is $K_{\sigma}$ (a countable union of compact subspaces),
 for every reflexive Banach space $E$. However, $C_p(X)$ is $K_{\sigma}$ if and only if $X$ is finite \cite[Theorem I.2.1]{Arch}.
\end{remark}

Our next statement, which is one of the main results of our paper, significantly extends Remark \ref{remark1}. 

\begin{theorem}\label{theorem3}
Let $E$ be a Banach space.
If there exist an infinite Tychonoff space $X$ and a homeomorphism $T: C_p(X) \to E_{w}$, then
\begin{itemize}
\item[{\rm (a)}] $X$ is a countable union of compact sets $X_n, n \in \omega$, where at least one component $X_n$ is non-scattered;
\item[{\rm (b)}] $E$ contains an isomorphic copy of the Banach space $\ell_{1}$.
\end{itemize}
\end{theorem}

A Banach space $E$ is called {\it weakly compactly generated} (WCG) if there is a weakly compact subset $K$ in $X$ such that
$E=\overline{{\span}(K)}$. WCG spaces constitute a large and important class of Banach spaces \cite{fabian}.
Every weakly compact subspace of a Banach space is called an {\it Eberlein} compact space.
 
\begin{corollary}\label{cor:main1}
Let $E$ be a WCG (separable) Banach space.
If there exist an infinite Tychonoff space $X$ and a homeomorphism $T: C_p(X) \to E_{w}$, then
\begin{itemize}
\item[{\rm (a)}] $X$ is a countable union of Eberlein (metrizable, respectively) compact spaces $X_n, n \in \omega$, where at least one component $X_n$ is non-scattered;
\item[{\rm (b)}] $E$ contains an isomorphic copy of the Banach space $\ell_{1}$.
\end{itemize}
\end{corollary}

\begin{corollary}\label{cor:main2}
Let $E$ be a WCG (separable) Banach space.
If there exist an infinite compact space $X$ and a homeomorphism $T: C_p(X) \to E_{w}$, then 
$X$ is a non-scattered Eberlein (metrizable, respectively) compact space and $E$ contains a copy of $\ell_{1}$.
\end{corollary}

Theorem \ref{theorem3} provides a generalization of \cite[Corollary 5.11, Theorem 5.12]{Krupski-2}, because
the Banach space $C(L)$ over a compact space $L$ contains an isomorphic copy of $\ell_{1}$ if and only if $L$ is non-scattered
(see \cite[Theorem 12.29]{fabian}). Note that our approach is different from that presented in \cite{Krupski-2};
we make use of the main result of \cite{okunev} and 
some standard arguments from general topology and functional analysis.

Necessary conditions in Theorem \ref{theorem3} are not sufficient,
because the Banach space $\ell_{1}$ is never homeomorphic to any space $C_p(X)$ (cf. Remark \ref{remark_l1}).
This and some other remarks related with the topic were mentioned earlier in \cite{KalMol}.

 Note that  M. Krupski and W. Marciszewski have asked already whether the spaces
$C_p(K)$ and $C(K)_w$ are homeomorphic for $K = [0,1]^{\omega}, K = \beta\mathbb{N}, K =\beta\mathbb{N} \setminus \mathbb{N}$
\cite[Questions 4.7 - 4.9]{Krupski-2}. 

Our notation is standard and follows the book \cite{En}. The closure of a set $A$ is denoted by $\overline{A}$.
The following very well-known notions play an important role in the proof of Theorem \ref{theorem3}.
A topological space $X$ is Fr\'echet-Urysohn if for each subset $A\subset X$ and each $x\in\overline{A}$ there exists a sequence $\{ x_{n}: n\in \omega\}$ in $A$ which converges to $x$.
 A subset $B$ of a topological vector space $E$ is \emph{bounded} in $E$ if for each neighbourhood of zero $U$  in $E$ there exists a scalar $\lambda$ such that $B \subset \lambda U$.
\section{Sequentially continuous linear mappings $T:C_p(X)\to E_w$}\label{section1}
 Let $X$ be a Tychonoff space. For a function $f:X\to \mathbb R$, we denote the support of $f$ by $$\supp(f)=\{t\in X: f(t) \neq 0\}.$$

In order to prove Theorem  \ref{Theor:general} we need an elementary lemma.
\begin{lemma}\label{lem1}
Let $X$ be an infinite  Tychonoff space and let $E$ be a Banach space. If  $T:C_p(X)\to E_w$ is a sequentially continuous linear operator,
 then for every  sequence $\{U_n: n\in \omega\}$  of pairwise disjoint nonempty open subsets of $X$ there exists $N$ such that  $T(f)=0$ for all $n\geqslant N$ and $f\in C(X)$ with $\supp(f)\subset U_n$.
\end{lemma}

\begin{proof}
On the contrary, suppose we can find a  strictly increasing sequence $\{k_n: n\in \omega\}$ of natural numbers and a sequence $\{f_n: n\in \omega\} \subset C(X)$ such that
$\supp(f_n)\subset U_{k_n}$ and $\|T(f_n)\| > 0$ for every $n$. Then the sequence $\Big\{n \frac{f_n}{\|T(f_n)\|}: n \in \omega \Big\}$ pointwise converges to zero, but the sequence of images
$$\Big\{n \frac{T(f_n)}{\|T(f_n)\|}: n \in \omega \Big\}$$ 
is not bounded in $E$.  We arrive at a contradiction with the facts that the families of bounded sets in $E_w$ and $E$ coincide 
and sequentially continuous linear operators map bounded sets into bounded sets.
\end{proof}

\begin{proof}[Proof of Theorem \ref{Theor:general}]\mbox{}

{\bf First step}: $X$ is a compact space.

Denote by $A$ the set of all points $t$ in $X$ such that there exists an open neighbourhood $U$ of $t$ in $X$
such that $T(f)=0$ for every $f\in C(X)$ satisfying the property $\supp(f)\subset U$. Evidently, the set $A$ is open.
We consider the following three cases:
\begin{enumerate}
\item The set $X\setminus A$ is infinite;
\item  The set $X\setminus A$ is empty;
\item  The set $X\setminus A$ is finite and nonempty.
\end{enumerate}

\medskip
\emph{Suppose that the item  (1) holds}.
\medskip

Since $X\setminus A$ is infinite, we can find a sequence $\{t_n: n\in \omega\} \subset X\setminus A$ and a sequence $\{U_n: n\in \omega\}$
 of pairwise disjoint open subsets of $X$ such that $t_n\in U_n$.
Then we choose $f_n\in C(X)$ such that $\supp(f_n)\subset U_n$ and $T(f_n) \neq 0$ for every $n$. This contradicts Lemma~\ref{lem1}.

\medskip
\emph{Suppose that the item (2) holds}.
\medskip

Let us fix $f\in C(X)$ for a moment.
 For every $t\in X$, we can find an open neighbourhood $U_{t}$ of $t$ in $X$ such that $T(g)=0$ for every $g\in C(X)$ satisfying $\supp(g)\subset U_{t}$.
 Since $X$ is a compact space, we choose a finite set  $\{t_{1},\ldots,t_{N}\}\subset X$ such that $X=\bigcup _{k=1}^{N} U_{t_{k}}$. 
Now we take the partition of unity subordinated to the open finite cover $\{U_{t_{k}}\}_{k=1}^{N}$ \cite[p. 300]{En}: 
  the functions $\{g_{1}\ldots,g_{N}\}\subset C(X)$ such that  $$\supp(g_{k})\subset U_{t_{k}}, \,\, g_{k}(X)\subset [0,1], \,\, \sum_{j=1}^{N} g_{j}(t)=1$$
for all $1\leqslant k\leqslant N$ and  $t\in X$. Note that $\supp(f g_{k})\subset U_{t_k}$ for every $k$, therefore,
$$T(f)=T\bigl(f\bigl(\sum_{k=1}^{N} g_{k}\bigr)\bigr)=\sum_{k=1}^{N} T\bigl(fg_{k}\bigr)=0.$$  Consequently,
$T(f)=0$ for every $f\in C(X)$, i.e. the range of $T$ is trivial.

\medskip
\emph{Suppose that the item  (3) holds.}
\medskip

Since $X\setminus A$ is finite, there exists a continuous linear extension operator $L:C_p(X\setminus A)\to C_p(X)$ such that
$L(f)|_{X\setminus A}=f$ for every $f\in C(X\setminus A)$.
Let us fix  $f\in C(X)$ for a moment.
For every $n$, we can find open sets $V_n$ and $W_n$  such that $$X\setminus A\subset V_n\subset  \overline{V_n}\subset W_n$$ and $$\bigl|\bigl(f-L(f|_{X\setminus A})\bigr)(t)\bigr|<\tfrac{1}{n}$$ for every $t\in W_n$.  For every $t\in X\setminus W_n$, we choose an open  neighbourhood $U_{t,n}$ of $t$ such that $U_{t,n}\subset X\setminus \overline{V_n}$ and 
 $T(g)=0$ for every $g\in C(X)$ satisfying $\supp(g)\subset U_{t,n}$. Since $X\setminus W_n$ is a compact space, we find a finite set $$\{t_{1,n},\ldots,t_{N_n,n}\}\subset X\setminus W_n$$ 
such that $$X\setminus W_n\subset \bigcup _{k=1}^{N_n} U_{t_{k,n}} \subset X\setminus \overline{V_n}.$$ 
  By standard arguments, we find a partition of unity $\{g_{1,n},\ldots,g_{N_n,n}\}\subset C(X)$ such that
$$\supp(g_{k,n})\subset U_{t_{k,n}},\,\,\,  g_{k,n}(X)\subset [0,1],\,\,\, \sum_{j=1}^{N_n} g_{j,n}(t)=1, \,\,\sum_{j=1}^{N_n} g_{j,n}(s)\leqslant 1 $$
for all $n$,  $1\leqslant k\leqslant N_n$, $t\in X\setminus W_n$ and $s\in X$. Therefore,
$$\biggl\|\bigl(f-L(f|_{X\setminus A})\bigr)-\bigl(f-L(f|_{X\setminus A})\bigr)\bigl(\sum_{k=1}^{N_n} g_{k,n}\bigr)\biggr\|\leqslant \tfrac{1}{n}$$ and  
  $$T\biggl(\bigl(f-L(f|_{X\setminus A})\bigr)\biggl(\sum_{k=1}^{N_n} g_{k,n}\biggr)\biggr)=
\sum_{k=1}^{N_n} T\bigl(\bigl(f-L(f|_{X\setminus A})\bigr)g_{k,n}\bigr)=0$$
for every $n$.  It is clear that the sequence $$\Big\{\bigl(f-L(f|_{X\setminus A})\bigr)\bigl(\sum_{k=1}^{N_n} g_{k,n}\bigr):n\in\omega \Big\}$$ converges to
$\bigl(f-L(f|_{X\setminus A})\bigr)$ in $C_p(X)$. Finally we deduce that
$$T(f)=T\bigl(L(f|_{X\setminus A})\bigr)$$
 for every $f\in C(X)$, i.e. the dimension of the range of $T$ does not exceed the finite size of $X\setminus A$.

{\bf Second step}: $X$ is any Tychonoff space.

Denote by $C^{*}_p(X)$ the linear subspace of $C_p(X)$ consisting of all bounded continuous functions on $X$.
Recall a well known fact that $C^{*}_p(X) $ is sequentially dense in $C_p(X)$. Indeed, if $f$ is any function in $C(X)$, then for each natural $n$ define $f_n \in C^{*}_p(X)$
by the rule: $f_n(x) = f(x)$ if $|f(x)| \leq n$; $f_n(x) = n$ if $f(x) \geq n$; $f_n(x) = -n$ if $f(x) \leq -n$. Clearly, $\{f_n: n\in\omega \} \subset C^{*}_p(X)$ 
and the sequence $\{f_n: n\in\omega \}$ pointwise converges to $f$.

Every function from $C^{*}_p(X)$ uniquely extends to a function from $C(\beta X)$, where $\beta X$ is the Stone-\v{C}ech compactification of $X$. Denote by $\pi$ the linear continuous operator
of restriction: $\pi: C_p(\beta X) \to C_p(X)$. The range of $\pi$ is the linear space $C^{*}_p(X)$. Now we consider the composition
$$T\circ \pi: C_p(\beta X) \to E_w.$$
Let $C$ be the range of the operator $T\circ \pi$.
Since $T\circ \pi$ is a sequentially continuous linear operator, and $\beta X$ is compact,  the range $C$ is finite-dimensional by the first step.
But $C$ coincides with the image $T(C^{*}_p(X))$. Since $C^{*}_p(X)$ is sequentially dense in $C_p(X)$ and $T$ is sequentially continuous, we get that $C$ is dense in the range of the operator $T$.
However, $C$ is finite-dimensional, hence complete, finally we conclude that the whole range of $T$ coincides with the finite-dimensional linear space $C$.
\end{proof}

\section{When $C_p(X)$ and $E_w$ are homeomorphic?}\label{section2}

Let $S$ be the convergent sequence, that is, the space homeomorphic to 
$\{0\} \cup \{\frac{1}{n}: n \in \mathbb{N}\}$.
It is known that for every compact metrizable space $X$
there exists a continuous surjection $T: C_{p}(S) \to C_{p}(X)$ \cite[Remark 3.4]{Kawamura}.
To keep our  paper self-contained we recall the argument in the proof of Proposition  \ref{prop2} below.
A topological space $X$ is called {\it analytic} if $X$ is a continuous image of the space $\mathbb{N}^{\mathbb{N}}$, 
which is in turn homeomorphic to the space of irrationals $J \subset \mathbb{R}$ (see e.g. \cite{analytic_sets}).

\begin{proposition}\label{prop2}
A locally convex space $E$ is analytic if and only if $E$ is a continuous image of the space $C_{p}(S)$.
\end{proposition}
\begin{proof}
Assume first that such a continuous mapping from $C_{p}(S)$ onto $E$ exists. Since $C_{p}(S)$ is separble and metrizable and $K_{\sigma\delta}$, hence analytic, the image $E$ is analytic as well.
 Conversely, assume that $E$ is analytic. We fix a continuous surjection $\varphi: \mathbb{N}^{\mathbb{N}}\rightarrow E$.
The space $C_p(S)$ is a $K_{\sigma\delta}$-subset of $\mathbb{R}^S$ but not $K_{\sigma}$. 
Hence, from the Hurewicz theorem (see \cite[Theorem 3.5.4]{analytic_sets}) it follows that $C_p(S)$ contains a closed copy $J$ of the space of irrationals.
Now we apply the classic Dugundji extension theorem (see \cite[Theorem 2.2]{Mill}), to get a continuous surjective mapping $T: C_p(S)\rightarrow E$ extending the mappping $\varphi$.
\end{proof}

\begin{corollary}\label{cor2} Let $E$ be a Banach space.
Then there exists a continuous surjective mapping $T: C_p(S) \rightarrow E_w$ if and only if $E$ is separable.
\end{corollary}
\begin{proof} If $E$ is a separable Banach space, then $E_w$ is analytic as a continuous image of a Polish space $E$.
It follows from Proposition \ref{prop2} that there exists a continuous surjection $T: C_p(S) \rightarrow E_w$.
Conversely, analytic space is separable, hence $E_w$ and $E$ are separable as well.
\end{proof}

Let $X$ be an infinite Tychonoff space, and $E$ be a Banach space.
For every continuous mapping $T:C_p(X)\to E_{w}$ we define the set $B(T)$ as follows: $B(T)$ consists of all points $t$ in $X$ such that there exists an open neighbourhood $U$ of $t$ in $X$ with the property
$$\sup\{\|T(f)\|:f\in C(X),\supp(f)\subset U\} < \infty.$$
Evidently, the set $B(T)$ is open.

The following lemma  will be used below.
\begin{lemma}\label{lem3}
Let $X$ be an infinite Tychonoff space and let $E$ be a Banach space. If  $T:C_p(X)\to E_w$ is a sequentially continuous map, then $X\setminus B(T)$  is finite.
\end{lemma}
\begin{proof}
 On the contrary, suppose that the set $X\setminus B(T)$ is infinite.

 \emph{Claim:
 There exist a sequence $\{t_n: n\in\omega\}\subset X\setminus B(T)$ and a sequence $\{U_n: n\in\omega\}$
 of pairwise disjoint open subsets of $X$ such that $t_n\in U_n$ for each natural $n$.}

 Indeed,  let us take any $s_1,s_2\in X\setminus B(T)$ such that $s_1\ne s_2$.   We find disjoint open neighbourhoods $V_1$ and $V_2$ of $s_1$ and $s_2$, respectively.
By the regularity of $X$, we find an open set $W_1$ such that $s_1\in W_1\subset \overline{W_1}\subset V_1$. At least one of the sets $(X\setminus B(T))\cap V_1$ and $(X\setminus B(T))\cap (X\setminus \overline{W_1})$ is infinite. If  the set $(X\setminus B(T))\cap V_1$ is infinite, we put $t_1=s_2$,  $U_1=V_2$ and $A_1=V_1$.
     If  the set $(X\setminus B(T))\cap V_1$ is finite and the set $$(X\setminus B(T))\cap (X\setminus \overline{W})$$ is infinite, we put $t_1=s_1$,
      $U_{1}=W_1$ and $A_1=X\setminus \overline{W_1}$.

          Suppose that for some natural $n$ we can find open sets $A_n, U_1,\ldots, U_n$ and points  $t_1,\ldots,t_n\in
     X\setminus B(T)$ such that the set  $(X\setminus B(T))\cap A_n$ is infinite,  $t_k\in U_k$, $A_n\cap U_k=\emptyset$ and  $U_m\cap U_j=\emptyset$ for all $1\leqslant k\leqslant n$ and $1\leqslant m<j\leqslant n$.
     We take any $$s_{2n+1},s_{2n+2}\in (X\setminus B(T))\cap A_n$$ such that $s_{2n+1}\ne s_{2n+2}$.   We have disjoint open neighbourhoods $V_{2n+1}$ and $V_{2n+2}$ of $s_{2n+1}$ and $s_{2n+2}$, respectively.

    By the regularity of $X$, we find an open set $W_{n+1}$ such that $$s_{2n+1}\in W_{n+1}\subset \overline{W_{n+1}}\subset A_n\cap V_{2n+1}.$$ 
		At least one of the sets $$(X\setminus B(T))\cap A_n\cap V_{2n+1}$$ and $$(X\setminus B(T))\cap A_n\cap (X\setminus \overline{W_{n+1}})$$ is infinite.
     If  the set $$(X\setminus B(T))\cap A_n\cap V_{2n+1}$$ is infinite, we put $t_{n+1}=s_{2n+2}$,  $U_{n+1}=V_{2n+2}$ and $A_{n+1}=A_n\cap V_{2n+1}$.
          If  the set $$(X\setminus B(T))\cap A_n\cap V_{2n+1}$$ is finite and the set $$(X\setminus B(T))\cap A_n \cap (X\setminus \overline{W_{n+1}})$$ is infinite,
					we put $t_{n+1}=s_{2n+1}$, $U_{n+1}=W_{n+1}$ and $A_{n+1}=A_n \cap (X\setminus \overline{W_{n+1}})$. An appeal to the mathematical induction completes the proof of the Claim.

         For every natural $n$, we find $f_n\in C_p(X)$ such that $\supp(f_n)\subset U_n$ and $\|T(f_{n+1})\|>\|T(f_n)\|+1$.
Then the sequence $\{f_n: n\in\omega\}$ pointwise converges to zero, hence $\{T(f_{n}): n\in\omega\}$ converges in $E_{w}$, but the sequence
$\{\|T(f_n)\|: n\in\omega\}$ is not bounded in $E$.  We arrive at a contradiction with the facts that the families of bounded sets in $E_w$ and $E$ coincide 
and sequentially continuous functions map convergent sequences into convergent sequences.
\end{proof}

The above results apply to get the following
\begin{proposition}\label{prop1}
Let $S$ be the convergent sequence and  let $E$ be an infinite-dimensional  separable Banach space. 
Then there is no a continuous linear surjection $T: C_{p}(S)\rightarrow E_w$ 
but  there  exists a continuous (non-linear) surjection $T: C_{p}(S)\rightarrow E_w$ such that $S\setminus B(T)$ is finite.
\end{proposition}
\begin{proof}
The first claim is an immediate consequence of Theorem \ref{Theor:general}.
The second claim is an immediate consequence of Corollary \ref{cor2} and Lemma \ref{lem3}.
\end{proof}

A mapping $T$ in Corollary \ref{cor2} is never a homeomorphism, because $C_p(S)$ is a metrizable and infinite-dimensional locally convex space,
while $E_w$ is metrizable provided it is finite-dimensional.   
In this section we are interested in finding necessary conditions for the existing of a homeomorphism $T:C_p(X)\to E_w$.
In order to find such conditions it is reasonable first to examine several basic topological properties that are satisfied by all infinite-dimensional spaces $E_w$.

\begin{remark}\label{remark3}\mbox{}\\
(1) {\it Countable chain condition (ccc)}.\\
It is a fundamental result on the weak topology, due to H.H. Corson, that $E_w$ satisfies the ccc property for every Banach space $E$ \cite[proof of Lemma 5]
{Corson}. 
However, $C_p(X)$ also always enjoys the ccc property \cite[Corollary 0.3.7]{Arch}, so we cannot distinguish these spaces by ccc.\\
(2) {\it Angelicity}.\\
Another fundamental result about weak topology is the Eberlein-\v{S}mulian theorem for every space $E_w$.
However, $C_p(X)$ also always enjoys this property for every compact space $X$ \cite{Arch}, so we cannot distinguish much
$E_w$ and $C_p(X)$ by angelicity.\\
(3) {\it Eberlein-Grothendieck property}.\\
A topological space $Z$ is called an {\it Eberlein-Grothendieck space} if
$Z$ homeomorphically embeds into the space $C_p(K)$ for some compact space $K$ (see \cite[p. 95]{Arch}).  
It is widely known that $E_w$ always embeds into $C_p(K)$, where $K$ is the compact unit ball of the dual $E^{\prime}$ endowed with the weak$^{\ast}$ topology,
i.e. $E_w$ always is an Eberlein-Grothendieck space.\\
(4) {\it $k$-space, sequentiality, Fr\'echet-Urysohn property}.\\
All the three properties coincide for each $C_p(X)$ by the Gerlits-Nagy theorem (see \cite{Arch}).
If $X$ is compact then $C_p(X)$ enjoys these properties if and only if $X$ is scattered.
Vice versa, if $E_{w}$ is  Fr\'echet-Urysohn, then $E$ is finite-dimensional.
Several alternative proofs are known for this statement.
a) By \cite[Lemma 14.6]{kak} the closed unit ball $B$  in $E$ is a $w$-neighbourhood of zero, so $E$ is finite-dimensional;
b) M. Krupski and W. Marciszewski \cite[Corollary 6.5]{Krupski-2} gave a simple proof for a stronger statement: $E_w$ is not a $k$-space, if $E$ is an infinite-dimensional
Banach space; and
c) Original proof of the latter fact appears in \cite[p. 280]{Sch}.
\end{remark}

We need several auxiliary results.

\begin{lemma}\label{lem_help1}
Let a Tychonoff space $X$ can be represented as a countable union of scattered compact sets.
Then the space $C_p(X)$ is Fr\'echet-Urysohn.
\end{lemma}
\begin{proof} Denote $X = \bigcup\{X_n: n \in \omega \}$, where each $X_n$ is a scattered compact space.
Define $Y_n$ to be the $\aleph_0$-modification of the topological space $X_n$, i.e. the family of all
 $G_{\delta}$-sets of $X_n$ is declared as a base of the topology of $Y_n$.
It is known that every $Y_n$ is a Lindel\"of $P$-space (see e.g. \cite[Lemma II.7.14]{Arch}).
 Define $Y$ to be the free countable union of all $Y_n$.
Evidently, $Y$ remains a Lindel\"of $P$-space. We define a natural continuous mapping $\varphi$ from $Y$ onto $X$ as follows.
Let $y \in Y$, then $y = x \in Y_n$ for a certain unique $n \in \omega$, and we define $\varphi(y) = x \in X$.
The map dual to $\varphi$ homeomorphically embeds $C_p(X)$ into $C_p(Y)$. The space $C_p(Y)$ is Fr\'echet-Urysohn (see e.g. \cite[Theorem II.7.15]{Arch}),
therefore the space $C_p(X)$ is Fr\'echet-Urysohn as well. 
\end{proof}

\begin{lemma}\label{lem_help2}
Let $X$ be a non-scattered compact space. Then 
for every finite set $A \subset X$ there is a non-scattered compact set $Y$ such that $Y \subset X \setminus A$.
\end{lemma}
\begin{proof}
We shall use the following classic theorem of A. Pe\l czy\'nski and Z. Semadeni:
Let $X$ be a compact space, then $X$ is scattered if and only if there is
 no continuous mapping of $X$ onto the segment $[0,1]$ (see \cite[Theorem~8.5.4]{Semadeni}).
Since $X$ is not scattered, there exists  a continuous surjection $f:X\to [0,1]$. 
Since the set $A$ is finite, we find a segment $[a,b]\subset [0,1]\setminus f(A)$.
Then $Y=f^{-1}([a,b])$ is a compact non-scattered subset of $X.$ 
\end{proof}

Now we are ready to present the proof of Theorem \ref{theorem3}.

\begin{proof}[Proof of Theorem \ref{theorem3}]
We have already observed in Remark \ref{remark3}
that $E_w$ always is an  Eberlein-Grothendieck space.
Hence, by our assumptions $C_p(X)$ is the image under a continuous open mapping of an Eberlein-Grothendieck space.
Making use of the fundamental result of O. Okunev, we immediately conclude that $X$ must be a $\sigma$-compact space (see \cite[Theorem 4]{okunev} or \cite[Corollary III.2.9]{Arch}).

Denote $X = \bigcup\{X_n: n \in \omega \}$, where each $X_n$ is a compact space. We claim that at least one component $X_n$ is non-scattered, and then clearly $X$ is non-scattered.
Indeed, otherwise, the space $C_p(X)$ would be Fr\'echet-Urysohn by Lemma \ref{lem_help1}, therefore also $E_w$ would be Fr\'echet-Urysohn, which is false, again by Remark \ref{remark3}.
Fix any $n$ such that the compact set $X_n$ is non-scattered.  

According to Lemma~\ref{lem3}, the set $X\setminus B(T)$ is finite, therefore we can apply Lemma \ref{lem_help2} and find a 
non-scattered compact set $Y \subset X_n \cap B(T)$. 
For every $t\in Y$, we choose open sets $V_t$ and $W_t$  such that $t\in V_t\subset  \overline{V_t}\subset W_t \subset B(T)$ and
$$\sup\{\|T(g)\|:g\in C(X),\supp(g)\subset W_t\} < \infty.$$

Since $Y$ is compact, we find finitely many sets $V_t$ covering $Y$.
There exists at least one $t \in Y$ such that $\overline{V_{t}}\cap Y$ is not scattered. Indeed, assuming that each $\overline{V_{t}}\cap Y$ is scattered we would get that
a non-scattered compact space $Y$ is covered by finitely many scattered compact sets, which is obviously impossible.
Fix a non-scattered $Z=\overline{V_{t}} \cap Y$. The next fact plays a crucial role: the space $C_p(Z)$ is not Fr\'echet-Urysohn. 
 We prove the following

\emph{Claim}: $F=\{f\in C(X):\supp(f)\subset W_{t}\}$
 \emph{is not a Fr\'echet-Urysohn space}.

Indeed, it is clear that $F$ is a closed subset of $C_p(X)$.
Let $G$ be a subset of $C_p(Z)$ satisfying the property: there exists $g\in \overline{G}$ such that
does not exist any sequence in $G$ which converges to $g$ in $C_p(X)$.
 Let
$$H=\{f\in F:f|_{Z}\in G\}.$$
Using the Tietze-Urysohn theorem and the Urysohn lemma we deduce that $\{f|_{Z}:f\in F\}=C_p(Z)$. 
Hence the space $\{f|_{Z}:f\in F\}$ fails the Fr\'echet-Urysohn property. To complete the proof  it is enough to show that
 $\{f|_{Z}:f\in \overline{H}\}=\overline{G}$. It is clear that $\{f|_{Z}:f\in \overline{H}\}\subset \overline{G}$.
Suppose that $$h\in \overline{G} \setminus \{f|_{Z}:f\in \overline{H}\}.$$ 
Let $\tilde{h}\in F$ be such that $\tilde{h}|_{Z}=h$. Since $\overline{H}$ is a closed subset of $C_p(X)$ 
and $\tilde h\notin \overline{H}$, we find $N\in \mathbb N$,  $\{t_1,\ldots,t_N\}\subset W_{t}$ and  $\varepsilon_j>0$ for every $1\leqslant j\leqslant N$ such that
 $$\{f\in F:|\tilde h(t_j)-f(t_j)|<\varepsilon_j, 1\leqslant j\leqslant N\}\cap \overline{H}=\emptyset.$$
Either
$\{t_1,\ldots,t_N\}\cap Z=\emptyset$ or $\{t_1,\ldots,t_N\}\cap Z\neq \emptyset$.
In the first case, according to the Tietze-Urysohn theorem for every $f\in G$ we find $\tilde{f}\in H$ such that 
$\tilde{f}|_{Z}=f$ and $\tilde{f}(t_j)=\tilde{h}(t_j)$
 for every $1\leqslant j\leqslant N$. Consequently, only the second case may hold. We may assume that there exists  $1\leqslant L\leqslant N$ such that
$$\{t_1,\ldots,t_N\}\cap Z= \{t_1,\ldots,t_L\}.$$ It is clear that
$$\{f|_{Z}: f \in F, |h(t_j)-f(t_j)|<\varepsilon_j, 1\leqslant j\leqslant L\}\cap G=\emptyset.$$
Consequently, $h\notin \overline{G}$. Thus we have arrived at a contradiction. The  Claim has been proved.

Finally, relying on the definition of $F$, we observe that $T(F)$ is a bounded subset of $E$ which is not a Fr\'echet-Urysohn space in the weak topology.
 According to \cite[Proposition 4.4]{Barroso}, the Banach space $E$ contains a subspace isomorphic to $l_1$, which finishes the proof.
\end{proof}

In \cite[Corollary 5.11, Theorem 5.12]{Krupski-2} M. Krupski and W. Marciszewski proved that for infinite compact spaces $K$ and $L$, where $L$ is scattered,
 the spaces $C_{p}(K)$,  $C(L)_{w}$  and the spaces $C_{p}(L)$, $C(K)_{w}$ are not homeomorphic.
Our Theorem \ref{theorem3} generalizes both results because the Banach space $C(L)$ does not contain a subspace isomorphic to $l_1$, if $L$ is scattered.

\begin{proof}[Proof of Corollary \ref{cor:main1}]
By Theorem \ref{theorem3} we have that $X = \bigcup\{X_n: n \in \omega \}$, where each $X_n$ is a compact space.
On the other hand, if $E$ is a WCG  Banach space, then $E_w$ contains a dense $\sigma$-compact subspace.
Therefore, $C_p(X)$ also contains a dense $\sigma$-compact subspace, which we denote by $Y$.
 For each $n\in\omega$ consider the restriction mapping $\pi_n$ from $C_p(X)$ onto 
 $C_p(X_n)$. It is easily seen that $Z_n = \pi_n(Y)$ is a dense $\sigma$-compact subspace of $C_p(X_n)$. 
It follows that a compact space $X_n$ is Eberlein  (see \cite[Theorem IV.1.7]{Arch}) for each $n$.
Assume now that $E$ is separable, then $E_w$ is also separable, hence analogously to the above case, $C_p(X_n)$ is separable for each $n$.
It follows that a compact space $X_n$ is metrizable  (see \cite[Theorem I.1.5]{Arch}) for each $n$.
The rest is provided by Theorem \ref{theorem3}.
\end{proof}

\begin{proof}[Proof of Corollary \ref{cor:main2}]
This is done actually in the previous proof. Just replace $X_n$ by $X$.
\end{proof}

\begin{remark}\label{remark_l1}
Necessary conditions in Theorem \ref{theorem3} are not sufficient,
because  the Banach space $E = \ell_{1}$ in the weak topology is never homeomorphic to any space $C_p(X)$.
The reason is the following: every separable Banach space $E$ with the Schur property is an $\aleph_0$-space in
the weak topology. But $C_p(X)$ is an $\aleph_0$-space if and only if $X$ is countable, i.e. if and only if $C_p(X)$ is metrizable.
(For the details see \cite{Gab2}).   
\end{remark}

\begin{remark}\label{remark_gulko}
The arguments used in the proof of Theorem \ref{theorem3} are not applicable for the question whether $X$ in that result
must be a compact space and not just a $\sigma$-compact space. This is because compactness of $X$ is not invariant under the homeomorphisms of the spaces $C_p(X)$.
For instance, the spaces $C_p[0,1]$ and $C_p(\mathbb{R})$ are homeomorphic \cite{gulko}.
\end{remark}

We finish the paper by the following challenging question.

\begin{problem}\label{prob5}
Does there exist a separable Banach space $E$ such that $E_w$ is homeomorphic to $C_p[0,1]$?
\end{problem}


\begin{thebibliography}{99}

\bibitem{Arch} A. V. Arkhangel'ski,
 \textit{Topological Function Spaces}, Kluwer, Dordrecht, 1992.

\bibitem{Barroso} C. S. Barroso, O. F. K. Kalenda, P. K. Lin,
 \textit{On the approximate fixed point property in abstract spaces}, Math. Z. \textbf{271} (2012), 1271--1285.

\bibitem{Corson} H. H. Corson,
\textit{The weak topology of a Banach space}, Trans. Amer. Math. Soc. \textbf{101} (1961), 1--15.

\bibitem{En} R.~Engelking,
\textit{General Topology}, Heldermann Verlag, Berlin, 1989.

\bibitem{fabian} M. Fabian, P. Habala, P. H\' ajek, V. Montesinos, J. Pelant, V. Zizler,
 \textit{Functional Analysis and Infinite-Dimensional Geometry}, CMS Books Math./Ouvrages Math. SMC, 2001.

\bibitem{Gab} S. Gabriyelyan, J. Grebik, J. K\c akol, L. Zdomskyy,
 \textit{The Ascoli property for function spaces}, Topology Appl. \textbf{214} (2016), 35--50.

\bibitem{Gab2} S. Gabriyelyan, J. K\c akol, W. Kubi\'s, W. Marciszewski,
 \textit{Networks for the weak topology of Banach and Fr\' echet spaces},
 J. Math. Anal. Appl. \textbf{432} (2015), 1183--1199. 

\bibitem{Gab3} S. Gabriyelyan, J. K\c akol, G. Plebanek, 
 \textit{The Ascoli property for function spaces and the weak topology on Banach and Fr\' echet spaces},
 Studia Math. \textbf{233} (2016), 119--139.

\bibitem{gulko} S. P. Gul’ko, T. E. Khmyleva,  
 \textit{Compactness is not preserved by the relation of $t$-equivalence},
Mathematical Notes of the Academy of Sciences of the USSR,
\textbf{39} (1986), 484--488. 

\bibitem{Kawamura} K. Kawamura, A. Leiderman,
 \textit{Linear continuous surjections of $C_p$-spaces over compacta}, Topology Appl. \textbf{227} (2017),  135--145.

\bibitem{kak} J. K\c akol, W. Kubi\'s, M.~Lopez-Pellicer,
  \textit{Descriptive Topology in Selected  Topics of Functional Analysis}, Developments in Mathematics, Springer, New York, 2011.

\bibitem{Kakol-Leiderman} J. K\c akol, A. Leiderman,
 \textit{On linear continuous operators between distinguished spaces $C_p(X)$},  to appear in RACSAM.

\bibitem{KalMol} J. K\c akol, S. Moll-L\'opez,
\textit{A note on the weak topology of spaces $C_k(X)$ of continuous functions},
RACSAM (2021) 115:125,
https://doi.org/10.1007/s13398-021-01051-1

\bibitem{Krupski-1} M. Krupski,
 \textit{On the weak and pointwise topologies in function spaces}, RACSAM \textbf{110} (2016), 557--563.


\bibitem{Krupski-2} M. Krupski, W. Marciszewski,
 \textit{On the weak and pointwise topologies in function spaces II}, J. Math. Anal. Appl. \textbf{452} (2017), 646--658.

\bibitem{Marciszewski} W. Marciszewski,
 \textit{Function Spaces}, in Recent Progress in General Topology II, Edited by M. Hu\v sek, J. van Mill, North-Holland (2002), 345--369.
 
\bibitem{Mill} J.\ van\ Mill,
 \textit{The Infinite-Dimensional Topology of Function Spaces}, North-Holland Mathematical Library \textbf{64}, North-Holland, Amsterdam, 2001.

\bibitem{okunev} O. G. Okunev,
\textit{Weak topology of an associated space and $t$-equivalence},
Mathematical Notes of the Academy of Sciences of the USSR, \textbf{46} (1989), 534--538. 

\bibitem{analytic_sets}  C. A. Rogers, J. E. Jayne,
\textit{$K$-analytic sets}, in: Analytic Sets, Academic Press, 1980, p. 1--181.

\bibitem{Sch} G. Schl\" uchtermann, R. F. Wheeler, 
\textit{The Mackey dual of a Banach space}, Note di Mat. \textbf{11} (1991), 273--287.

\bibitem{Semadeni} Z. Semadeni,
\textit{Banach spaces of continuous functions}, Volume I,
PWN - Polish Scientific Publishers, Warszawa, 1971.
\end{thebibliography}
\end{document}